\documentclass[graybox, draft]{svmult}

\usepackage{helvet}  
\usepackage{courier} 
\usepackage{type1cm} 
\usepackage{makeidx}  
\usepackage{graphicx} 
\usepackage{multicol} 
\usepackage[bottom]{footmisc}

\usepackage{amssymb,amsmath,graphicx}
\DeclareMathOperator{\n}{\mathfrak{n}}

\DeclareMathOperator{\p}{p}

\DeclareMathOperator{\T}{\mathfrak{t}}

\DeclareMathOperator{\Hom}{Hom}
\DeclareMathOperator{\End}{End}
\DeclareMathOperator{\Aut}{Aut}
\DeclareMathOperator{\Der}{Der}
\DeclareMathOperator{\Ker}{Ker}
\DeclareMathOperator{\Img}{Im}
\DeclareMathOperator{\id}{Id}

\DeclareMathOperator{\spaned}{span}

\usepackage{algpseudocode, algorithm, algorithmicx}

\begin{document}

\title*{Derivations and automorphisms of free nilpotent Lie algebras and their quotients}
\titlerunning{Derivations and automorphisms of nilpotent Lie algebras}
\author{Pilar Benito and Jorge Rold\'an-L\'opez}
\institute{Pilar Benito \at Universidad de La Rioja, Dpto. Matem\'aticas y Computaci\'on, Edificio CCT - Madre de Dios 53, 26006 Logro\~no, La Rioja, Spain; \email{pilar.benito@unirioja.es}
\and Jorge Rold\'an-L\'opez \at Universidad de La Rioja, Dpto. Matem\'aticas y Computaci\'on, Edificio CCT - Madre de Dios 53, 26006 Logro\~no, La Rioja, Spain; \email{jorge.roldanl@unirioja.es}}
\maketitle

\abstract{
Let $\n_{d,t}$ be the free nilpotent Lie algebra of type $d$ and nilindex $t$. Starting out with the derivation algebra and the  automorphism group of $\n_{d,t}$, we get a natural description of derivations and automorphisms of any generic nilpotent Lie algebra of the same type and nilindex. Moreover, along the paper we discuss several examples to illustrate the obtained results.
}

\section{Introduction}
\label{sec:1}

In the middle of the 20th century, the study of derivations and automorphisms of algebras was a central topic of research. It is well known that many linear algebraic Lie groups and their Lie algebras arise from the automorphism groups and the derivation algebras of certain nonassociative algebras. In fact, for a given finite-dimensional real nonassociative algebra $A$, the automorphism group $\Aut A$ is a closed Lie subgroup of the lineal group $\mathtt{GL}\mathrm{(}A\mathrm{)}$ and the derivation algebra $\Der A$ is the Lie algebra of $\Aut A$ (see~\cite[Proposition 7.1 and 7.3, Chapter 7]{SaWa-73}).

Paying attention to Lie algebras, a lot of research papers on this topic are devoted to the study of the interplay between the structures of their derivation algebras, their groups of automorphisms and Lie algebras themselves (see~\cite{VaVa-06} and references therein). We point out two simple but elegant results on this direction. According to \cite{BoSe-53}, any Lie algebra that has an automorphism of prime period without nonzero fixed points is nilpotent. The same result is valid in the case of Lie algebra has a nonsingular derivation (see~\cite[Theorem~2]{Ja-55}). So, automorphisms and derivations and the nature of their elements are interesting tools in the study of structural properties of algebras.

The main motif of this paper is to describe the group of automorphisms and the algebra of derivations of any finite-dimensional $t$-step nilpotent Lie algebra $\n$ (this means that $\n^t\neq 0=\n^{t+1}$) generated by a set $U$ of $d$ elements. The description will be given through the derivation algebra and the automorphism group of the free $t$-step nilpotent Lie algebra $\n_{d,t}$ generated by $U$. Denoting by $\mathfrak{u}=\spaned\langle U\rangle$, the elements of the derivation algebra, $\Der \n_{d,t}$, arise by extending and combining, in a natural way, linear maps from $\mathfrak{u}$ into $\mathfrak{u}$ and from $\mathfrak{u}$ to $\n_{d,t}^2$. The group of automorphisms, $\Aut \n_{d,t}$, is described through automorphism induced by elements of the general linear group $\mathtt{GL}(\mathfrak{u})$ and automophisms provide by linear maps from $\mathfrak{u}$ to $\n_{d,t}$ which induce the identity mapping on $\dfrac{\n_{d,t}}{\n_{d,t}^2}$.

The paper splits into three sections starting from number 2. Section~2 collects some known results about derivations and automorphisms of free nilpotent Lie algebras. This information let establish the structure of derivations and automorphisms of any nilpotent Lie algebra in Theorems~\ref{Tder} and~\ref{Taut}. Section~3 contains examples which show the way to compute automorphism groups and derivation algebras. This last section may also be consider as a short illustration of techniques which may be used in this regard.

Along the paper, vector spaces are of finite dimension over a field $\mathbb{F}$ of characteristic zero. All unexplained definitions may be found in \cite{Ja-62} or \cite{Hu-72}.

\section{Theoretical results}
\label{sec:2}

We begin by recalling some basics facts and notations about Lie algebras. Let $\n$ be a Lie algebra with bilinear product $[x,y]$, and $\Aut \n$ denote the automorphism group of $\n$, that is, the set of linear bijective maps $\varphi\colon\n\to \n$ such that $\varphi[x,y]=[\varphi(x),\varphi(y)]$. Also, let $\Der \n$ denote the set of derivations of $\n$, so, the set of linear maps $d\colon A\to A$ such that $d[x,y]=[d(x),y]+[x,d(y)]$. If $V, W$ are subspaces of $\n$, $[V,W]$ denotes the space spanned by all products $[v,w]$, $v\in V, w\in W$. The terms of the lower central series of $\n$ are defined by $\n^1=\n$, and $\n^{i+1}=[\n,\n^i]$ for $i\geq 2$. If $\n^t\neq 0$ and $\n^{t+1}=0$, then $\n$ is said nilpotent of nilpotent index or nilindex $t$. We refer to them also as $t$-step nilpotent. A nilpotent algebra $\n$ is generated by any set $\{y_1,\dots,y_d\}$ of $\n$ such that $\{y_i+\n^2:i=1,\dots,d\}$ is a basis of $\frac{\n}{\n^2}$ (see \cite[Corollary 1.3]{Ga-73}). The dimension of this space is called the type of $\n$ and $\{y_1,\dots,y_d\}$ is said \emph{minimal set of generators} (m.s.g.). So, the type is just the dimension of any subspace $\mathfrak{u}$ such that $\n=\mathfrak{u}\oplus[\n,\n]$.

The free $t$-step nilpotent Lie algebra on the set $U=\{x_1,\dots,x_d\}$ (where $d\geq 2$) is the quotient algebra $\n_{d,t}=\mathfrak{FL}(U)/\mathfrak{FL}(U)^{t+1}$, where $\mathfrak{FL}(U)$ is the free Lie algebra generated by $U$ (see~\cite[Section~4, Chapter~V]{Ja-62}). The elements of $\mathfrak{FL}(U)$ are linear combinations of monomials 
$[x_{i_1},\dots, x_{i_s}]=[\dots[[x_{i_1},x_{i_2}],x_{i_3}],\dots, x_{i_s}]$, $s\geq 1$ and $x_{i_j}\in U$. So, the free nilpotent algebra $\n_{d,t}$ is generated as vector space by  $s$-monomials $[x_{i_1},\dots, x_{i_s}]$, for $1\leq s\leq t$.

Again, if we set $\mathfrak{u}=\spaned\langle U\rangle$, the subspace $\mathfrak{u}^s=[\mathfrak{u}^{s-1},\mathfrak{u}]$ is the linear span of the $s$-monomials. Thus $\n_{d,t}$ is an $\mathbb{N}$-graded algebra whose $s$-th homogeneous component is $\mathfrak{u}^s$. The dimension of any subspace $\mathfrak{u}^s$, $1\leq s\leq t$ is:
\begin{equation*}
\frac{1}{s}\sum_{a\mid s}\mu(a)d^{s/a},
\end{equation*}where $\mu$ is the M\"oebius function.

The algebra $\n_{d,t}$ enjoys the following \emph{Universal Mapping Property} (see~\cite[Proposition 1.4]{Ga-73} and \cite[Proposition 4]{Sa-71}): for any $k$-step nilpotent Lie algebra~$\n$ with $k\leq t$ of type $d$, and any $d$-elements $y_1,\dots,y_d$ of $\n$, the correspondence $x_i\mapsto y_i$ extends to a unique algebra homomorphism $\n_{d,t}\to \n$. In the particular case that $\{y_1,\dots, y_d\}$ is a m.s.g., the image contains a set of generators, so the map is surjective. Therefore, any $t$-step nilpotent Lie algebra of type $d$ is an homomorphic image $\dfrac{\n_{d,t}}{\T}$ where $\T$ is an ideal such that $\T\subseteq \n_{d,t}^2$ and $\n_{d,t}^t \not \subseteq \T$.

Derivations and automorphisms of $\n_{d,t}$ are completely determine by their effect on $\mathfrak{u}$. Conversely, any linear map from $\mathfrak{u}$ into $\n_{d,t}$ (bijection from a basis of $\mathfrak{u}$ to any m.s.g.) determines a unique derivation (automorphism) of $\n_{d,t}$. This assertion is covered by the next result and its corollary. A detailed proof can be found in \cite[Propositions 2 and 3]{Sa-71}.

\begin{proposition}\label{der-aut-free} Let $\varphi$ denote any linear map from the vector space $\mathfrak{u}=\spaned \langle x_1,\dots,x_d\rangle$ into $\n_{d,t}$, where $\{x_1,\dots, x_d\}$ is a m.s.g. of $\n_{d,t}$. Then: 
\begin{enumerate}
\item[a)] $\varphi$ extends to a derivation of $\n_{d,t}$ by declaring$$d_\varphi([x_{\alpha_1},\dots,x_{\alpha_r}])=\sum_{1\leq i\leq r}[x_{\alpha_1},\dots, \varphi(x_{\alpha_i}),\dots,x_{\alpha_r}].$$
\item[b)] $\varphi$ extends to an algebra homomorphism of $\n_{d,t}$ by declaring $$\Phi_\varphi([x_{\alpha_1},\dots,x_{\alpha_r}])=[\varphi(x_{\alpha_1}),\dots,\varphi(x_{\alpha_r})].$$
\end{enumerate}
Moreover if $\p_\mathfrak{u}$ stands for the projection map from $\n_{d,t}$ into $\mathfrak{u}$, then $\Phi_\varphi$ is an automorphism iff $\{\p_\mathfrak{u}(\varphi(x_1)),\dots, \p_\mathfrak{u}(\varphi(x_n))\}$ is a linearly independent set.
\end{proposition}

\begin{corollary} Let $\n_{d,t}$ be the free $t$-nilpotent Lie algebra on $d$-generators $x_1,\dots, x_d$ and $\mathfrak{u}=\spaned\langle x_i\rangle$. The derivation algebra and the automorphism group of $\n_{d,t}$ are described as 
$\Der {\n_{d,t}}=\{d_\varphi: \varphi\in \Hom(\mathfrak{u},{\n_{d,t}})\}$ and $\Aut {\n_{d,t}}=\{\Phi_\varphi: \varphi\in \Hom(\mathfrak{u},{\n_{d,t}})\text{ and }\{\p_\mathfrak{u}(\varphi(x_1)), \dots, \p_\mathfrak{u}(\varphi(x_d))\} \textrm{  m.s.g.}\}$.
\end{corollary}

\begin{remark} The Levi factor $\mathfrak{S}_{d,t}$ of $\Der \n_{d,t}$ is given by the maps $d_\varphi$ for $\varphi \in \mathfrak{sl}(\mathfrak{u})$. Clearly, $\mathfrak{S}_{d,t}$ is isomorphic to the special Lie algebra $\mathfrak{sl}_d(\mathbb{F})$. The elements of the nilpotent radical $\mathfrak{N}_{d,t}$ are the linear maps $d_\varphi$ where $\varphi \in \Hom(\mathfrak{u},\n_{d,t}^2)$. And the solvable radical is just $\mathfrak{R}_{d,t}=k\cdot id_{d,t}\oplus \mathfrak{N}_{d,t}$ where $id_{d,t}(a_k)=k\cdot a_k$ for any $a_k\in \mathfrak{u}^k$ (see \cite[Proposition 2.4]{BeCo-14}). 
\end{remark}

\begin{remark} The group $\Aut \n_{d,t}$ is the semidirect product of the general linear group $\mathtt{GL}(d,t)$, obtained from the automorphisms $\Phi_\varphi$ where $\varphi \in GL(\mathfrak{u})$, and the nilpotent group $\mathtt{NL}(d,t)$, whose elements are $\Phi_\sigma$ and $\sigma=\id_\mathfrak{u}+\delta$ and $\delta \in \End(\mathfrak{u},n_{d,t}^2)$ (see \cite[Proposition 3.1]{BeCoLa-17}). 
\end{remark}

For any ideal $\T$ of $\n_{d,t}$ such that $\n_{d,t}^t \not \subseteq \T\subseteq \n_{d,t}^2$, let denote by $\Der_{\T} {\n_{d,t}}$ and $\Der_{\n_{d,t}, \T} \n_{d,t}$ the subset of derivations which map $\T$ into itself, and $\n_{d,t}$ into $\T$ respectively. Both sets are subalgebras of $\Der \n_{d,t}$, even more, $\Der_{\n_{d,t},\T} \n_{d,t}$ is an ideal inside $\Der_{\T} \n_{d,t}$, and the following result follows \cite[Proposition 5]{Sa-71}:

\begin{theorem}\label{Tder}
Let $\T$ be an ideal of $\n_{d,t}$ such that $\n_{d,t}^t \not \subseteq \T\subseteq \n_{d,t}^2$, the algebra of derivations of $\dfrac{\n_{d,t}}{\T}$ is isomorphic to $\dfrac{\Der_{\T} \n_{d,t}}{\Der_{\n_{d,t}, \T} \n_{d,t}}$, where $\Der_{\T} \n_{d,t}$ and $\Der_{\n_{d,t}, \T} {\n_{d,t}}$ maps $\T$ and $\n_{d,t}$ into $\T$ respectively.
\end{theorem}
 
In a similar vein to the previous theorem, it is possible to arrive at a structural description of automorphisms of homomorphic images of free nilpotent algebras. For any ideal $\T$ of $\n_{d,t}$, $\n_{d,t}^t \not \subseteq \T\subseteq \n_{d,t}^2$, let denote by $\Aut_{\T} {\n_{d,t}}$ the subset of automorphisms which map $\T$ into itself. It is easily checked that $\Aut_{\T} {\n_{d,t}}$ is a subgroup of $\Aut {\n_{d,t}}$. Consider now the map:
\begin{equation*}
\theta\colon \Aut_{\T} {\n_{d,t}} \to \Aut \frac{\n_{d,t}}{\T}, \quad \theta(\Phi)(x+\T)=\Phi(x)+\T
\end{equation*}
By using $\Phi(\T)=\T$ and $\Phi$ homomorphism, we can easily check that $\theta$ is well defined.
Now, a straightforward computation shows that $\theta$ is a group homomorphism with kernel,
\begin{equation*}
\Ker \theta=\{\Phi\in\Aut \n_{d,t}: \Img(\Phi-Id)\subseteq \T\}.
\end{equation*}
Then, we have the following result:

\begin{theorem}\label{Taut}
For any ideal $\T$ of $\n_{d,t}$ such that $\n_{d,t}^t \not \subseteq \T\subseteq \n_{d,t}^2$, the set $\Aut_{\T}^\circ \n_{d,t}=\{\Phi\in\Aut \n_{d,t}: \Img(\Phi-Id)\subseteq \T\}$ is a normal subgroup of the group of automorphisms of $\dfrac{\n_{d,t}}{\T}$. Moreover $\Aut \dfrac{\n_{d,t}}{\T}$ is isomorphic to $\dfrac{\Aut_{\T} \n_{d,t}}{\Aut_{\T}^\circ \n_{d,t}}$, where $\Aut_{\T} \n_{d,t}$ maps $\T$ into $\T$.
\end{theorem}

\begin{proof}
From previous comments we only need to proof that the map $\theta$ is onto. Let $\rho_\mathfrak{t}: \n_{d,t}\to \frac{\n_{d,t}}{\T}$ be the canonical projection and let $\{f_1+\T, \ldots, f_k+\T\}$ be a basis of $\frac{\n_{d,t}}{\T}$ and $\{e_1 + \T,\ldots, e_d + \T\}$ a m.s.g. of $\frac{\n_{d,t}}{\T}$. Then $\{e_1, \ldots, e_d\}$ is also a m.s.g. of $\n_{d,t}$. If we take a generic automorphism $\hat{A}\in \Aut \dfrac{\n_{d,t}}{\T}$, 
\begin{equation*}
    \hat{A}(e_i + \T) = \sum_{j=1}^k \alpha_{ij} f_j + \T,
    \ \text{and}\ \text{declare} \  
    A(e_i) = \sum_{j=1}^k \alpha_{ij} f_j,
\end{equation*}
$A$ extends to a linear homomorphism, $A\colon\mathfrak{e}\to \n_{d,t}$, where $\mathfrak{e}=\spaned\langle e_1, \ldots, e_d\rangle$. Let $\Phi_A$ be the homomorphism given by Proposition \ref{der-aut-free}. We check that $\theta(\Phi_A) = \hat{A}$ noting that, for a generic element $[[\ldots[a_1,a_2],\ldots,a_l]$ where $a_i \in \mathfrak{e}$, up to linear combinations, we have that
\begin{equation*}
\begin{split}
    \rho_{\T} \circ \Phi_A [[\ldots[a_1,a_2],\ldots,a_l] &= [[\ldots[\rho_{\T} \circ A(a_1),\rho_{\T} \circ A(a_2)],\ldots,\rho_{\T} \circ A(a_l)] \\&= [[\ldots[\hat{A}\circ \rho(a_1),\hat{A}\circ \rho_{\T}(a_2)],\ldots,\hat{A}\circ \rho_{\T}(a_l)] \\&= \hat{A}\circ \rho_{\T}[[\ldots[a_1,a_2],\ldots,a_l].
\end{split}
\end{equation*}
The second equality follows because for every $a_i = \sum_{j=1}^d\beta_{ji}e_j$,
\begin{equation*}
\begin{split}
    \rho_{\T}& \circ A(a_i) = \rho_{\T}\circ A \left(\sum_{j=1}^d\beta_{ji}e_j\right) = \sum_{j=1}^d\beta_{ji}\rho_{\T}\circ A(e_j) = \sum_{j=1}^d\beta_{ji}\rho_{\T}\left( \sum_{l=1}^k\alpha_{jl} f_l\right)
    \\&=\sum_{j=1}^d\beta_{ji} \sum_{l=1}^k(\alpha_{jl} f_l+\T) = \sum_{j=1}^d\beta_{ji} \hat{A}(e_i + \T) = \sum_{j=1}^d\beta_{ji} \hat{A}\circ \rho_{\T}(e_i) = \hat{A}\circ\rho_{\T}(a_i).
\end{split}
\end{equation*}
Now $\Ker\, \rho_{\T}=\T$ and $\rho_{\T} \circ \Phi_A=\hat{A} \circ \rho_{\T}$ implies $\Phi_A(\T)=\T$ and, $\Phi_A$ automorphism, follows by using the equivalence given in Proposition \ref{der-aut-free} and the fact that $\hat{A}$ is an  automorphism.
\end{proof}

\section{Examples, techniques and patterns}
\label{sec:3}

 From the generator set $U=\{x_1,\dots,x_d\}$, we easily get the standard monomials $[x_{i_1},\dots,x_{i_r}]$ that (linearly) generate the Lie algebra $\n_{d,t}$. However, the anticommutativity law ($[x_i,x_j]+[x_j,x_i]=0$) and the Jacobi identity ($\sum_{\text{cyclic}}[[x_i,x_j],x_k]=0$), both set linear dependence relations. This makes it difficult to find a basis formed by monomials. The problem was solved by M. Hall in 1950. Focusing on the behavior of algorithms, the most natural basis to work on free nilpotent Lie algebras, is the \emph{Hall basis} (see~\cite{Ha-50} for definition, and \cite[Chapter~IV, Section 5]{Se-92} for a detailed construction).

Starting with the total order $x_d<x_{d-1}<\dots<x_1$, the definition of Hall basis states recursively if a given standard monomial depends on the previous ones. The recursive algorithm is covered by the pseudocode given in Table~\ref{tab:1} and provides a Hall basis that we will denote as $H_{d,t}(U_{<})$ or $H_{d,t}$ if the total order in $U$ is clear. This algorithm checks if an element $v$ belongs to the Hall basis once we have defined a monomial order. For some small $d$ and $t$ values, the output of Hall basis algorithm is given in Table~\ref{tab:2}.

\begin{table}
\caption{{\bf:}$\quad$ Hall basis algorithm}
\label{tab:1} 
\begin{tabular}{p{0.992\linewidth}}
\noalign{\smallskip}\hline\noalign{\smallskip}
\textit{isCanonical}($v$):
	    
	    \setlength{\parindent}{15pt}
	    \textbf{if} $\deg v == 1$ \textbf{ then true};
	    
	    \textbf{else if} (\textbf{not }\textit{isCanonical}($v_1$)\textbf{ or }\textbf{not  }\textit{isCanonical}($v_2$)\textbf{ or }$v_2 > v_1$)\textbf{ then false};
	    
	    \textbf{else if} $\deg v_1 > 1$\textbf{ then }(\textit{isCanonical}($v_{1,1}$)\textbf{ or }\textit{isCanonical}($v_{1,2}$)\textbf{ or }$v_2 \geq v_{1,2}$);
	    
	    \textbf{else true};\\
\noalign{\smallskip}\hline\noalign{\smallskip}
\end{tabular}
$^a$ Note that this is a recursive algorithm. Here $v = [v_1, v_2]$. In order to generate Hall Basis elements of degree $n$ we can combine $v_1$ and $v_2$ in level $n-k$ and $k$ respectively, where $k = 1, \ldots, n/2$.
\end{table}
\begin{table}
\caption{{\bf:}$\quad$ Hall basis of $\n_{d,t}$}
\label{tab:2} 
\fontsize{8.2pt}{10pt}\selectfont
\begin{tabular}{p{0.67cm}p{10.83cm}}
\hline\noalign{\smallskip}
$(d,t)$ & $H_{d,t}$ \\
$(2, 6)$&$x_2, x_1, [x_1, x_2], [[x_1, x_2], x_2], [[x_1, x_2], x_1], [[[x_1, x_2], x_2], x_2], [[[x_1, x_2], x_2], x_1],$\\
&$ [[[x_1, x_2], x_1], x_1], [[[[x_1, x_2], x_2], x_2], x_2], [[[[x_1, x_2], x_2], x_2], x_1], $\\
&$ [[[x_1, x_2], x_2], [x_1, x_2]],[[[[x_1, x_2], x_2], x_1], x_1], [[[x_1, x_2], x_1], [x_1, x_2]],$\\
&$ [[[[x_1, x_2], x_1], x_1], x_1],[[[[[x_1, x_2], x_2], x_2], x_2], x_2], [[[[[x_1, x_2], x_2], x_2], x_2], x_1], $\\
&$ [[[[x_1, x_2], x_2], x_2], [x_1, x_2]],[[[[[x_1, x_2], x_2], x_2], x_1], x_1], [[[[x_1, x_2], x_2], x_1], [x_1, x_2]],$\\
&$ [[[[[x_1, x_2], x_2], x_1], x_1], x_1],[[[x_1, x_2], x_1], [[x_1, x_2], x_2]], [[[[x_1, x_2], x_1], x_1], [x_1, x_2]], $\\
&$ [[[[[x_1, x_2], x_1], x_1], x_1], x_1] $\\
\noalign{\smallskip}
$(4, 3)$&$x_4, x_3, x_2, x_1, [x_3, x_4], [x_2, x_4], [x_2, x_3], [x_1, x_4], [x_1, x_3], [x_1, x_2], [[x_3, x_4], x_4],$\\
& $[[x_3, x_4], x_3], [[x_3, x_4], x_2], [[x_3, x_4], x_1], [[x_2, x_4], x_4], [[x_2, x_4], x_3], [[x_2, x_4], x_2],$\\
& $ [[x_2, x_4], x_1], [[x_2, x_3], x_3], [[x_2, x_3], x_2], [[x_2, x_3], x_1], [[x_1, x_4], x_4], [[x_1, x_4], x_3],$\\
&$[[x_1, x_4], x_2], [[x_1, x_4], x_1], [[x_1, x_3], x_3], [[x_1, x_3], x_2], [[x_1, x_3], x_1], [[x_1, x_2], x_2], $\\
& $[[x_1, x_2], x_1] $\\
\noalign{\smallskip}
\!$(6, 2)$&$x_6, x_5, x_4, x_3, x_2, x_1, [x_5, x_6], [x_4, x_6], [x_4, x_5], [x_3, x_6], [x_3, x_5], [x_3, x_4], [x_2, x_6],$\\
& $[x_2, x_5], [x_2, x_4], [x_2, x_3], [x_1, x_6], [x_1, x_5], [x_1, x_4], [x_1, x_3], [x_1, x_2] $\\
\noalign{\smallskip}
\noalign{\smallskip}\hline\noalign{\smallskip}
\end{tabular}
\footnotesize
$^b$ We point out that from expanded basis $H_{4,3}$ and $H_{2,6}$ we can recover Hall basis of $\n_{4,2}$ and $\n_{2,t}$ for $t=2,3,4,5$.
\end{table}

\noindent Now we introduce several examples which ilustrate (among other things):
\begin{enumerate}
\item{The way to describe a generic $d$-generated t-nilpotent Lie algebra as an homomorphic image of $\n_{d,t}$.}
\item{The way to compute automorphisms and derivations regarding Proposition~1 and Theorems 1 and 2.}
\item{The recognition of some structural patterns of nilpotent algebras depending on the nature of their derivations and automorphisms.}
\end{enumerate}
In the sequel, if a map $\varphi$ is given in a matrix form $A=(a_{ij})$ attached to a basis $\mathcal{B}=\{v_1,\dots,v_n\}$, then $\varphi(v_i)=\sum_{j=1}^n a_{ji}v_j$.

The Universal Mapping Property lets us describe any $t$-nilpotent Lie algebra $\n$ of type $d$ as a homomorphic image of $\n_{d,t}$ in a easy way. From any m.s.g. $\{e_1,\dots,e_d\}$ of $\n$, the correspondence $x_i\mapsto e_i$ for $i=1,\dots, d$ extends uniquely to a surjective algebra homomorphism $\theta_{\n}\colon\n_{d,t}\to \n$ and $\n\cong \dfrac{\n_{d,t}}{\Ker \theta_{\n}}$. We will compute ideals of this type in our following example.

\begin{example}\label{ex:dix-heis} Let $\n_1$ and $\n_2$ be the $8$-dimensional and $5$-dimensional Lie algebras described through the basis $\{e_1,\dots,e_8\}$ and $\{u_1,\dots,u_5\}$ by the following multiplication table ($[a,b]=-[b,a]$ and $[a,b]=0$ is not in the table):
\vspace{-0.5in}
\begin{multicols}{3}
\begin{align*}
    [e_{1},e_{2}] &= e_{5},\\
    [e_{1},e_{3}] &= e_{6},\\
    [e_1,e_4]     &= e_{7},\\
    [e_1,e_5]     &= -e_{8},
\end{align*}

\begin{align*}
    [e_2,e_3]     &= e_{8},\\
    [e_{2},e_{4}] &= e_{6},\\
    [e_{2},e_{6}] &= -e_{7},\\
    [e_3,e_4]     &= -e_{5},
\end{align*}

\begin{align*}
    [e_3,e_5]     &= -e_{7},\\
    [e_4,e_6]     &= -e_{8},\\
    [u_1,u_3]     &= u_5,\\
    [u_2,u_4]     &= u_5,
\end{align*}
\end{multicols}
\noindent The lower central series of these algebras are:
\begin{equation*}
\n_1^2=\spaned\langle e_5,e_6,e_7,e_8\rangle,\quad \n_1^3=\spaned\langle e_7,e_8\rangle,\quad \n_1^4=0,
\end{equation*}
and
\begin{equation*}
\n_2^2=\spaned\langle u_5\rangle,\quad \n_2^3=0.
\end{equation*}
Consider now the maps $\theta_{\n_1}\colon x_i\to e_i$ for $i=1,\dots,4$ from $\n_{4,3}$ onto $\n_1$ and $\theta_{\n_2}\colon x_i\to u_i$ for $i=1,\dots, 4$ from $\n_{4,2}$ onto $\n_2$. Both correspondances extend to homomorphisms of algebras as in the proof in \cite[Proposition 4]{Sa-71} ($\theta[x_{\alpha_1}\dots x_{\alpha_s}]=[\theta(x_{\alpha_1})\dots \theta(x_{\alpha_s})]$). It is not hard to see that:

\begin{equation*}
\begin{split}
\Ker \theta_{\n_1}=\spaned\langle [x_3, x_4]+[x_1, x_2], [x_2, x_4]-[x_1, x_3], [x_2, x_3]-[[x_1, x_3], x_1],&\\
[x_1, x_4]+[[x_1, x_2], x_2], [[x_3, x_4], x_4]-[[x_1, x_3], x_1], [[x_3, x_4], x_3],&\\
[[x_3, x_4], x_2]+[[x_1, x_2], x_2],[[x_3, x_4], x_1], [[x_2, x_4], x_4], [[x_1, x_4], x_2],&\\
[[x_2, x_4], x_3]+[[x_1, x_2], x_2], [[x_2, x_4], x_2], [[x_2, x_4], x_1]-[[x_1, x_3], x_1],&\\
[[x_2, x_3], x_3], [[x_2, x_3], x_2], [[x_2, x_3], x_1], [[x_1, x_4], x_4],[[x_1, x_4], x_3],&\\
[[x_1, x_4], x_1],[[x_1, x_3], x_3]+[[x_1, x_2], x_2],[[x_1, x_3], x_2], [[x_1, x_2], x_1]\rangle,&
\end{split}
\end{equation*}
and
\begin{equation*}
\Ker \theta_{\n_2}=\spaned\langle [x_3, x_4],[x_2, x_3],[x_1, x_4],[x_1, x_2],[x_1, x_3]-[x_2, x_4]\rangle .
\end{equation*}
We point out that $\Ker \theta_{\n_2}$ is an homogeneous ideal in the $\mathbb{N}$-graded structure of $\n_{4,2}$ and $\Ker \theta_{\n_1}$ is not an homogeneous ideal  of~$\n_{4,3}$. Therefore, $\n_2$ inherits the grading of $\n_{4,2}$, but $\n_1$ does not inherit that of $\n_{4,3}$.
\end{example} 

\begin{example}\label{ex:n24} According to Proposition \ref{der-aut-free}, derivations and automorphisms of $\n_{2,4}$ in Hall basis $H_{2,4}$ can be easily obtained by iterating the Leibniz rule $\varphi([a,b])=[\varphi(a),b])]+[a,\varphi(b)]$ and the law $\varphi([a,b])=[\varphi(a),\varphi(b)]$. The matrices representing the elements of $\mathrm{Aut} \n_{2,4}=\mathtt{GL}(2,4)\rtimes \mathtt{NL}(2,4)$ are product of matrices of the following shapes:
\begin{equation*}
{\scriptsize \left (\begin{array} {c|c|c|c} \begin{array}{cc}a_1&a_2\\a_3&a_4\end{array} & 0 & 0 & 0\\ \hline 0& \epsilon & 0 & 0 \\ \hline 0& 0& \begin{array}{cc}\epsilon a_1&\epsilon a_2\\\epsilon a_3&\epsilon a_4\end{array}& 0 \\ \hline 0& 0 & 0&\epsilon\cdot A'\end{array} \right )\in \mathtt{GL}(2,4),
\left( \begin{array} {c|c|c|c} I_2 & 0 & 0 & 0\\ \hline \begin{array}{cc}b_1&b_2\end{array}& 1 & 0 & 0 \\ \hline \begin{array}{cc}c_1&c_2\\c_3&c_4\end{array}& \begin{array}{c}b_2\\-b_1\end{array}& I_2 & 0 \\ \hline \begin{array}{cc}d_1&d_2\\d_3&d_4\\d_5&d_6\end{array}& \begin{array}{c}c_2\\c_4-c_1\\-c_3\end{array} & \begin{array}{cc}b_2&0\\-b_1&b_2\\0&-b_1\end{array}& I_3\end{array} \right )\in \mathtt{NL}(2,4); 
}\end{equation*}
here $I_k$ denotes the $k\times k$ identity matrix, $\epsilon=a_1a_4-a_2a_3\neq 0$ and
$${\scriptsize A'=\left( \begin{array}{ccc}a_1^2&a_1a_2&a_2^2\\2a_1a_3&a_1a_4+a_2a_3&2a_2a_4\\a_3^2&a_3a_4&a_4^2\end{array}\right ).}
$$From the decomposition $\Der \n_{2,4}=\mathfrak{S}_{2,4}\oplus \mathbb{F}\cdot id_{2,4}\oplus \mathfrak{N}_{2,4}$, the matrices that represent derivations of $\n_{2,4}$ are sum of matrices of three different types:
\begin{equation*}
{\scriptsize \left (\begin{array} {c|c|c|c} \begin{array}{cc}a_1&a_2\\a_3&-a_1\end{array} & 0 & 0 & 0\\ \hline 0& 0 & 0 & 0 \\ \hline 0& 0& \begin{array}{cc}a_1&a_2\\a_3&-a_1\end{array}& 0 \\ \hline 0& 0 & 0&\begin{array}{ccc}2a_1&a_2&0\\2a_3&0&2a_2\\0&a_3&-2a_1\end{array}\end{array} \right )\in \mathfrak{S}_{2,4},\  
\lambda id_{2,4}=\left( \begin{array} {c|c|c|c} \begin{array}{cc}\lambda&0\\0&\lambda\end{array} & 0 & 0 & 0\\ \hline 0& 2\lambda & 0 & 0 \\ \hline 0& 0&\begin{array}{cc}3\lambda&0\\0&3\lambda\end{array} & 0 \\ \hline 0& 0 & 0& \begin{array}{ccc}4\lambda&0&0\\0&4\lambda&0\\0&0&4\lambda\end{array}\end{array} \right )}\end{equation*}
and 
\begin{equation*}
{\scriptsize \left( \begin{array} {c|c|c|c} 0 & 0 & 0 & 0\\ \hline \begin{array}{cc}b_1&b_2\end{array}& 0 & 0 & 0 \\ \hline \begin{array}{cc}c_1&c_2\\c_3&c_4\end{array}& \begin{array}{c}b_2\\-b_1\end{array}& 0 & 0 \\ \hline \begin{array}{cc}d_1&d_2\\d_3&d_4\\d_5&d_6\end{array}& \begin{array}{c}c_2\\c_4-c_1\\-c_3\end{array} & \begin{array}{cc}b_2&0\\-b_1&b_2\\0&-b_1\end{array}& 0\end{array} \right )\in \mathfrak{N}_{2,4}.}\end{equation*}

For any $0\neq \lambda \in \mathbb{F}$, the linear map $\varphi_\lambda (x_i)=\lambda x_i$ provides the (semisimple) automorphism $\Phi_{\varphi_\lambda} ([x_{\alpha_1}\dots x_{\alpha_r}])=\lambda^r[x_{\alpha_1}\dots x_{\alpha_r}]$ and the (semisimple) derivation $d_{\varphi_{\lambda}}([x_{\alpha_1}\dots x_{\alpha_r}])=r\lambda[x_{\alpha_1}\dots x_{\alpha_r}]$.

\end{example}

Consider now the $5$-dimensional Lie algebra $\n_3$ with basis $\{z_1,\dots,z_5\}$ and nonzero products: 
\begin{equation*}
[z_1,z_2]=z_3,\quad [z_1,z_3]=z_4,\quad [z_1,z_4]=[z_2,z_3]=z_5.
\end{equation*}
The lower central series is $\n_3^2=\spaned\langle z_3,z_4,z_5\rangle$, $\n_3^3=\spaned\langle z_4,z_5\rangle$, $\n_3^4=\spaned\langle z_5\rangle $ and $\n_3^5=0$. So, the correspondence $x_i\mapsto z_i$ for $i=1,2$ extends to a surjective algebra homomorphism $\theta_{\n_3}\colon\n_{2,4}\to \n_3$ and $\n_3\cong \dfrac{\n_{2,4}}{\Ker \theta_{\n_3}}$. In this case, the kernel is the $3$-dimensional ideal:
\begin{multline*}
\Ker \theta_{\n_3}=\spaned\langle [[[x_1, x_2], x_2], x_2], [[[x_1, x_2], x_2], x_1],\\
[[x_1, x_2], x_2]+[[[x_1, x_2], x_1], x_1]\rangle.
\end{multline*}

\begin{example}\label{ex:John}
Let denote $\mathfrak{t}=\Ker \theta_{\n_3}$. According to Theorems \ref{Tder} and \ref{Taut}, derivations (automophisms) of $\n_3$ are a quotient of the set of derivations (automorphisms) of $\n_{2,4}$ that leave $\mathfrak{t}$ invariant. These sets are:

\begin{equation*}
{\scriptsize \Der_\mathfrak{t} \mathfrak{n}_{2,4}: \left(
\def\arraystretch{1.25}
\begin{array} {c|c|c|c} \begin{array}{cc}a_1&a_2\\0&\frac{1}{2}a_1\end{array} & 0 & 0 & 0\\ \hline \begin{array}{cc}b_1&b_2\end{array}& \frac{3}{2}a_1 & 0 & 0 \\ \hline \begin{array}{cc}c_1&c_2\\c_3&c_4\end{array}& \begin{array}{c}b_2\\-b_1\end{array}& \begin{array}{cc}\frac{5}{2}a_1&a_2\\0&2a_1\end{array}& 0 \\ \hline \begin{array}{cc}d_1&d_2\\d_3&d_4\\d_5&d_6\end{array}& \begin{array}{c}c_2\\c_4-c_1\\-c_3\end{array} & \begin{array}{cc}b_2&0\\-b_1&b_2\\0&-b_1\end{array}&\begin{array}{ccc}\frac{7}{2}a_1&a_2&0\\0&3a_1&2a_2\\0&0&\frac{5}{2}a_1\end{array}\end{array} \right ),}
\end{equation*}
and, for $a_4\neq 0$,
\begin{equation*}
{\scriptsize \Aut_\mathfrak{t} \mathfrak{n}_{2,4}: 
\left(
\def\arraystretch{1.3}
\begin{array}{cc|c|cc|ccc}
 a_4^2 & a_2 & 0 & 0 & 0 & 0 & 0 & 0 \\
 0 & a_4 & 0 & 0 & 0 & 0 & 0 & 0 \\\hline
 b_1 & b_2 & a_4^3 & 0 & 0 & 0 & 0 & 0 \\\hline
 c_1 & c_2 & a_4^2 b_2-a_2 b_1 & a_4^5 & a_2 a_4^3 & 0 & 0 & 0 \\
 c_3 & c_4 & -a_4 b_1 & 0 & a_4^4 & 0 & 0 & 0 \\\hline
 d_1 & d_2 & a_4^2 c_2-a_2 c_1 & a_4^4 b_2-a_2 a_4^2 b_1 & a_2 \left(a_4^2 b_2-a_2 b_1\right) & a_4^7 & a_2 a_4^5 & a_2^2 a_4^3 \\
 d_3 & d_4 & c_4 a_4^2-a_4 c_1-a_2 c_3 & -a_4^3 b_1 & a_4^3 b_2-2 a_2 a_4 b_1 & 0 & a_4^6 & 2 a_2 a_4^4 \\
 d_5 & d_6 & -a_4 c_3 & 0 & -a_4^2 b_1 & 0 & 0 & a_4^5
\end{array}
\right).
}
\end{equation*}

Note that the isomorphism $\frac{\mathfrak{n}_{2,4}}{\mathfrak{t}}\to\n_3$ is provided by the correspondence $z_i'\mapsto z_i$ by taking $z_1'=x_1+\mathfrak{t}$, $z_2'=x_2+\mathfrak{t}$, $z_3'=[x_1,x_2]+\mathfrak{t}$, $z_4'=[x_1,[x_1,x_2]]+\mathfrak{t}$, $z_5'=[x_2,[x_1,x_2]]+\mathfrak{t}$. So $\mathcal{B'}=\{z_1', z_2', z_3',z_4', z_5'\}$ is a basis. Now, by using the isomorphisms in Theorem \ref{Tder} and Theorem \ref{Taut} and a minor change of basis, we get a complete description of derivations and automorphisms of $\n_3\cong\dfrac{\mathfrak{n}_{2,4}}{\mathfrak{t}}$. Relative to the basis $\{z_1,z_2,z_3,z_4,z_5\}$:
\begin{equation*}
{\scriptsize \Der \frac{\mathfrak{n}_{2,4}}{\mathfrak{t}}: \left (\begin{array} {c|c|c|c} \begin{array}{cc}a_1&0\\a_3&2a_1\end{array} & 0 & 0 & 0\\ \hline \begin{array}{cc}b_1&b_2\end{array}& 3a_1 & 0 & 0 \\ \hline \begin{array}{cc}c_1&c_2\end{array}& \begin{array}{c}b_2\end{array}& 4a_1& 0 \\ \hline \begin{array}{cc}d_1&d_2\end{array}&c_2-b_1& a_3+b_2& 5a_1\end{array} \right ),}
\end{equation*}
and, for $a_4\neq 0$,
\begin{equation*}
{\scriptsize \Aut \frac{\mathfrak{n}_{2,4}}{\mathfrak{t}}:\left(
\def\arraystretch{1.3}
\begin{array}{cc|c|c|c}
 a_4 & 0 & 0 & 0 & 0 \\
 a_2 & a_4^2 & 0 & 0 & 0 \\\hline
 b_2 & b_1 & a_4^3 & 0 & 0 \\\hline
 -c_4 & -c_3 & a_4 b_1 & a_4^4 & 0 \\\hline
 d_6-c_2 & d_5-c_1 & a_2 b_1-a_4 (a_4 b_2+c_3) & a_4^2 (a_2 a_4+b_1) & a_4^5
\end{array}
\right).}
\end{equation*}

From previous descriptions, it is clear that the map $\varphi_\lambda\colon x_i\to \lambda x_i$, for $i=1,2$, extends to a derivation iff $\lambda=0$ and $\varphi_\lambda$ extends to an automorphism iff $\lambda=1$. We also remark that, $\mathfrak{t}$ is not an homogeneous ideal, so $\n_3$ does not inherit the natural $\mathbb{N}$-grading of $\n_{2,4}$. However $\Phi_\lambda\colon x_i\to \lambda^ix_i$ is an automorphism for all $0\neq \lambda \in \mathbb{F}$ with eigenvalues $\lambda^i$ for $1\leq i\leq 5$. In the case of $\mathbb{F}$ be the reals and $\lambda>1$, $\Phi_\lambda$ is an (expanding) automorphism that provides the $\mathbb{N}$-grading $\n_3=\oplus_{i=1}^5 S(\lambda^i)$ where $S(\lambda^i)=\{v\in\n_3:  \Phi_\lambda(v)=\lambda^iv\}$.
\end{example}

As in the previous example, in the final one we get the conditions that determine derivations and automorphisms of $\n_2$, the Lie algebra described in Example \ref{ex:dix-heis}, by using $\Der \n_{4,2}$ and $\Aut_{4,2}$.

\begin{example}\label{Heisenberg} 
Let now $\mathfrak{t}=\Ker \theta_{\n_2}$. Derivations and automorphisms of $\n_{4,2}$ in Hall basis $H_{4,2}$ are (here $\Delta_{i,j}^{k,l}=a_ia_j-a_ka_l$):
\begin{equation*}
\arraycolsep=1.2pt
\Der \n_{4,2}: {\scriptsize d_A\!=\!\left(
\begin{array}{cccc|cccccc}
 a_{1} & a_{2} & a_{3} & a_{4} & 0 & 0 & 0 & 0 & 0 & 0 \\
 a_{5} & a_{6} & a_{7} & a_{8} & 0 & 0 & 0 & 0 & 0 & 0 \\
 a_{9} & a_{10} & a_{11} & a_{12} & 0 & 0 & 0 & 0 & 0 & 0 \\
 a_{13} & a_{14} & a_{15} & a_{16} & 0 & 0 & 0 & 0 & 0 & 0 \\\hline
 b_{1} & b_{2} & b_{3} & b_{4} & a_{1}+a_{6} & a_{7} & -a_{3} & a_{8} & -a_{4} & 0 \\
 b_{5} & b_{6} & b_{7} & b_{8} & a_{10} & a_{1}+a_{11} & a_{2} & a_{12} & 0 & -a_{4} \\
 b_{9} & b_{10} & b_{11} & b_{12} & -a_{9} & a_{5} & a_{6}+a_{11} & 0 & a_{12} & -a_{8} \\
 b_{13} & b_{14} & b_{15} & b_{16} & a_{14} & a_{15} & 0 & a_{1}+a_{16} & a_{2} & a_{3} \\
 b_{17} & b_{18} & b_{19} & b_{20} & -a_{13} & 0 & a_{15} & a_{5} & a_{6}+a_{16} & a_{7} \\
 b_{21} & b_{22} & b_{23} & b_{24} & 0 & -a_{13} & -a_{14} & a_{9} & a_{10} & a_{11}+a_{16} \\
\end{array}
\right),
}\end{equation*}
and, for nonsingular matrices with entries $a_i$,
\begin{equation*}
\def\arraystretch{1.5}
\Aut \n_{4,2}: {\scriptsize \Phi_A\!=\!\left(
\begin{array}{cccc|cccccc}
 a_{1} & a_{2} & a_{3} & a_{4} & 0 & 0 & 0 & 0 & 0 & 0 \\
 a_{5} & a_{6} & a_{7} & a_{8} & 0 & 0 & 0 & 0 & 0 & 0 \\
 a_{9} & a_{10} & a_{11} & a_{12} & 0 & 0 & 0 & 0 & 0 & 0 \\
 a_{13} & a_{14} & a_{15} & a_{16} & 0 & 0 & 0 & 0 & 0 & 0 \\\hline
 b_{1} & b_{2} & b_{3} & b_{4} & \Delta_{1,6}^{2,5} &\Delta_{1,7}^{3,5} & \Delta_{2,7}^{3,6} & \Delta_{1,8}^{4,5} & \Delta_{2,8}^{4,6} & \Delta_{3,8}^{4,7} \\
 b_{5} & b_{6} & b_{7} & b_{8} & \Delta_{1,10}^{2,9} &\Delta_{1,11}^{3,9} & \Delta_{2,11}^{3,10} & \Delta_{1,12}^{4,9} & \Delta_{2,12}^{4,10} & \Delta_{3,12}^{4,11} \\
 b_{9} & b_{10} & b_{11} & b_{12} & \Delta_{5,10}^{6,9} &\Delta_{5,11}^{7,9} & \Delta_{6,11}^{7,10} & \Delta_{5,12}^{8,9} & \Delta_{6,12}^{8,10} & \Delta_{7,12}^{8,11} \\
 b_{13} & b_{14} & b_{15} & b_{16} & \Delta_{1,14}^{2,13} & \Delta_{1,15}^{3,13} & \Delta_{2,15}^{3,14} & \Delta_{1,16}^{4,13} & \Delta_{2,16}^{4,14} & \Delta_{3,16}^{4,15} \\
 b_{17} & b_{18} & b_{19} & b_{20} & \Delta_{5,14}^{6,13} & \Delta_{5,15}^{7,13} & \Delta_{6,15}^{7,14} & \Delta_{5,16}^{8,13} & \Delta_{6,16}^{8,14} & \Delta_{7,16}^{8,15} \\
 b_{21} & b_{22} & b_{23} & b_{24} & \Delta_{9,14}^{10,13} & \Delta_{9,15}^{11,13} & \Delta_{10,15}^{11,14} & \Delta_{9,16}^{12,13} & \Delta_{10,16}^{12,14} & \Delta_{11,16}^{12,15} \\
\end{array}
\right).
}
\end{equation*}

\noindent Let $d_A\in \Der \n_{4,2}$ be and $\Phi_A\in \Aut \n_{4,2}$. An easy computation shows that
\begin{equation*}
d_A\in \Der_{\mathfrak{t}} \n_{4,2}\ \mathrm{iff}\ \left\{ 
\begin{array}{ll}
a_{12}= -a_{5},&a_{13}= a_{10},\ a_{7}= a_{4},\\
a_{15}= -a_{2},&a_{16}= a_{1}-a_{6}+a_{11},
\end{array}\right.
\end{equation*}
and
\begin{equation*}
\def\arraystretch{1.5}
\Phi_A\in \Aut_{\mathfrak{t}} \n_{4,2}\ \mathrm{iff}\ \left\{
\begin{array}{l}
\Delta_{1,10}^{2,9}+ \Delta_{5,14}^{6,13}=\Delta_{2,11}^{3,10}+\Delta_{6,15}^{7,14}=0,\\
\Delta_{1,12}^{4,9}+\Delta_{5,16}^{8,13}=\Delta_{3,12}^{4,11}+\Delta_{7,16}^{8,15}=0,\\
\Delta_{1,11}^{3,9}+\Delta_{2,12}^{4,10}+\Delta_{5,15}^{7,13}+\Delta_{6,16}^{8,14}=0.
\end{array}
\right.
\end{equation*}
Therefore, the correspondance $u_i\mapsto \lambda u_i$ for $i=1,\dots, 4$, and $u_5\mapsto 2\lambda u_5$ extends by linearity to a derivation of $\n_2$ for all $\lambda$. The correspondence $u_i\mapsto \lambda u_i$ for $i=1,\dots, 4$, and $u_5\mapsto \lambda^2 u_5$ extends to an automorphism if $\lambda\neq 0$.
\end{example}

\subsection*{\emph{Epilogue }}

In 1955, N. Jacobson proved in \cite[Theorem 3]{Ja-55} that any Lie algebra of characteristic zero with a nonsingular derivation is nilpotent. The author also noted that the vality of the converse was an open question. Two years later, J. Dixmier and W.G. Lister supplied in \cite{DiLis-57} a negative answer to the question by means of the algebra $\n_1$ that we have revisited in Example \ref{ex:dix-heis}. Every derivation of $\n_1$ is nilpotent, so the elements of $\Der \n_1$ are nilpotent maps, and therefore, $\Der \n_1$ is a nilpotent Lie algebra. It can be also proved that $\Aut \n_1$ is not a nilpotent group (see \cite{LeLu-72}). The existence of $\n_1$ is the starting point of the study of the so called \emph{characteristically nilpotent Lie algebras}, that is, Lie algebras in which any derivation is nilpotent. Over fields of characteristic zero, this class of algebras matches to the class of algebras in which every semisimple automorphism is of finite order (see \cite[Theorem~3]{LeTo-59}) or the class of algebras in which the algebra of derivations is nilpotent (see \cite[Theorem~1]{LeTo-59}).

\emph{Quasi-cyclic Lie algebras} were introduced at \cite{Le-63} by G. Leger in 1963. A nilpotent Lie algebra $\n$ is called quasi-cyclic (also known in the literature as homogenous) if $\n$ has a subspace $\mathfrak{u}$ such that $\n$ decomposes as the direct sum of subspaces $\mathfrak{u}^k=[\mathfrak{u}^k,\mathfrak{u}]$; in particular, quasi-cyclic algebras are $\mathbb{N}$-graded. Free nilpotent Lie algebras $\n_{d,t}$ are examples of this type of algebras. It is not hard to see that a nilpotent Lie algebra $\n\cong \frac{\n_{d,t}}{\mathfrak{t}}$ is quasi-cyclic iff $\mathfrak{t}$ is a homogeneous ideal of $\n_{d,t}$. In fact, quasi-cyclic Lie algebras are the class of nilpotent Lie algebras that contain a m.s.g. $\{e_1,\dots, e_d\}$ such that the correspondence $e_i\mapsto e_i$ extends to a derivation of $\n$ according to \cite[Corollary 1]{Jo-75}. By reviewing $\Ker \theta_{\n_i}$, we conclude that $\n_2$ is quasi-cyclic, but $\n_1$ and $\n_3$ are not. 

An automorphism of a real Lie algebra is called \emph{expanding authomorphism} if it is a semisimple automorphism whose eigenvalues are all greater than $1$ in absolute value. In 1970, J. L. Dyer states in \cite{Dy-70} that quasi-cyclic Lie algebras admits expanding automorphisms, the converse is false. The Lie algebra $\n_3$ provides a counterexample: according to Example \ref{ex:John}, $\n_3$ admits expanding automorphisms, but it is not quasi-cyclic. The algebra $\n_3$ has been introduced in \cite{Le-63} and \cite{Jo-75} to illustrate the results therein. The latter paper includes the characterization of (real) quasi-cyclic Lie algebras as those algebras that admit \emph{grading automorphisms}.

As it is noted in 1974 by J. Scheuneman in \cite[Section 1]{Sch-74}, \emph{``\emph{($\dots$)} the Lie algebra of a simply transitive group of affine motions of $\mathbb{R}^n$ has an affine structure \emph{($\dots$)}"}. The main result in this paper is that each $3$-step nilpotent Lie algebra has an complete (also known as transitive) \emph{affine structure}. Therefore, any homomorphic image $\frac{\n_{d,3}}{\mathfrak{t}}$ can be endowed with a such structure. The notion of (complete or transitive) \emph{affine structure} on Lie algebras is equivalent to that of (complete) \emph{left-symmetric structure} (see \cite{ElMy-94} and references therein). In fact any positively $\mathbb{Z}$-graded real Lie algebra admits a complete left symmetric structure according to \cite[Theorem 3.1]{DeLe-03}. Therefore, any quasi-cyclic nilpotent Lie algebra admits a complete left symmetric structure. There is an interesting interplay among gradings, expanding and hyperbolic automorphisms and affine structures.

It is not difficult to find recent research on derivations and automorphisms algebras and their applications. For nilpotent Lie algebras we point out \cite{LaOs-14}, \cite{OzEk-17} and \cite{SaSh-18}. So, this research area deserves to be considered.

\begin{acknowledgement}
The authors are partially funded by grant MTM2017-83506-C2-1-P of Ministerio de Econom\'ia, Industria y Competititividad (Spain). The second-named author is supported by a predoctoral research grant of Universidad de La Rioja.
\end{acknowledgement}


\begin{thebibliography}{99.}

\bibitem{BoSe-53} A. Borel, J.P. Serre, \textit{Sur certains sous-groupes des groupes de Lie compacts}, Comment. Math. Helv. \textbf{17} (1953), pp. 128--139.

\bibitem{DeLe-03} K. Dekimpe, K.B.Lee, \textit{Expanding maps, Anosov diffeomorphisms
and affine structures on infra-nilmanifolds}, Topology and its Applications \textbf{130} (2003) 259-269.

\bibitem{BeCo-14} P. Benito, D. de-la-Concepci\'on, \textit{An overview on free nilpotent Lie algebras}, Comment. Math. Univ. Carolina \textbf{55} (2014), no. 3, 325–339.

\bibitem{BeCoLa-17} P. Benito, D. de-la-Concepci\'on, J. Laliena \textit{Free nilpotent and nilpotent quadratic Lie algebras}, Linear Algebra Appl. \textbf{519} (2017), 296–326.

\bibitem{DiLis-57} J. Dixmier, W.G. Lister, \textit{Derivations of nilpotent Lie algebras}, Proc. Amer. Math. Soc. \textbf{8} (1957), pp. 155--158.

\bibitem{Dy-70} J.L. Dyer, \textit{A nilpotent Lie algebra with nilpotent automorphism group}, Bull. Amer. Math. Soc. \textbf{76} (1970), pp. 52--56.

\bibitem{ElMy-94} A. Elduque, H. Myung, \textit{On transitive left-symmetric algebras}, Nonassociative algebra and its applications (Oviedo, 1993), Math. Appl. \textbf{303} Kluwer Acad. Publ. (1994).

\bibitem{Ga-73} M. Gauger, \textit{On the classification of metabelian Lie algebras}, Trans. Amer. Math. Soc. \textbf{179} (1973), pp. 293--329.

\bibitem{Ha-50} Hall, M., \textit{A basis for free Lie rings and higher commutators in free groups}, Proc. Amer. Math. Soc. 1 (1950): 575-581.

\bibitem{Hu-72} J.E. Humphreys, \textit{Introduction to Lie algebras and representation theoryLie algebras}, (Springer-Verlag, NY, 1972).

\bibitem{Ja-55} N. Jacobson, \textit{A note on automorphisms and derivations of Lie algebras}, Proc. Amer. Math. Soc. \textbf{6} (1955), pp. 281--283.

\bibitem{Ja-62} N. Jacobson, \textit{Lie algebras}, (Dover Publications, NY, 1962).

\bibitem{Jo-75} R.W. Johnson, \textit{Homogeneous Lie algebras and expanding automorphisms}, Proc. Amer. Math. Soc. \textbf{48} (1975), pp. 292--296.

\bibitem{LaOs-14} J. Lauret et D. Oscari, \textit{On non-singular 2-step nilpotent Lie algebras}, Math. Res. Lett. \textbf{21} (2014), no. 03, 553--583.

\bibitem{LeLu-72} G. Leger, E. Luks: \textit{On nilpotent groups of algebra automophisms}, Nagoya Math. J. \textbf{46} (1972), pp. 87--95.

\bibitem{LeTo-59} G. Leger, S. T$\hat{o}$g$\hat{o}$: \textit{Characteristically nilpotent Lie algebras}, Duke Math. J. \textbf{26} (1959), pp. 623--628.

\bibitem{Le-63} G. Leger, \textit{Derivations of Lie algebras III}, Duke Math. J. \textbf{30} (1963), pp. 637--645.

\bibitem{OzEk-17} O. $\ddot{O}$ztekin, N. Ekici, \textit{Central automorphisms of free nilpotent Lie algebras}, J. Algebra Appl. \textbf{16} (2017), no. 11, 1750205, 8 pp.

\bibitem{SaSh-18} S. Saeedi, S. Sheikh-Mosheni, \textit{Derivation subalgebras of Lie algebras}, Note Mat. \textbf{38} (2018), no. 2, 105–115.


\bibitem{SaWa-73} Sagle, A.A., Walde R.E., \textit{Introduction to Lie groups and Lie algebras}, (Academic Press, NY, 1973).

\bibitem{Sa-71} T. Sato, \textit{The derivations of the Lie algebras}, Tohoku Math. \textbf{23} (1971),21--36.

\bibitem{Sch-74} J. Scheuneman, \textit{Affine structures on nilpotent Lie algebras}, Proc. Amer. Math. Soc. \textbf{46}, number 3, (1974), pp. 451--454.

\bibitem{Se-92} J.P. Serre, \textit{Lie algebras and Lie groups: 1964 Lectures given at Harvard University, 2nd edition}, (Springer-Verlag, Berlin Heidelberg, 1992)

\bibitem{VaVa-06} V. R. Varea, J.J. Varea, \textit{On Automorphisms and Derivations of a Lie Algebra}, Algebra Colloquium \textbf{13}: 1 (2006) 119--132.

\end{thebibliography}
\end{document}